\documentclass[11pt]{amsart}
\usepackage{amsfonts}
\usepackage{amssymb,latexsym}
\usepackage{enumerate}
\usepackage{mathrsfs}
\usepackage{amsmath}
\usepackage{amsthm}
\usepackage{hyperref}
\usepackage[square, comma, sort&compress, numbers]{natbib}
\usepackage{bm}
\usepackage{fancyhdr}
\usepackage{lastpage}
\usepackage{layout}
\usepackage{ulem}
\usepackage{extarrows}
\newtheorem{theorem}{Theorem}[section]
\newtheorem{proposition}[theorem]{Proposition}
\newtheorem{remark}[theorem]{Remark}
\newtheorem{lemma}[theorem]{Lemma}
\newtheorem{definition}[theorem]{Definition}
\numberwithin{equation}{section}
\newtheorem{corollary}[theorem]{Corollary}
\DeclareMathOperator \Vol{Vol}
\DeclareMathOperator \Herm{Herm}

\DeclareMathOperator \Id{Id}
\DeclareMathOperator \YM{YM}
\DeclareMathOperator \rank{rank}

\begin{document}

\title{\textsc{Numerically flat holomorphic bundles over  non K\"ahler manifolds }}
\author{Chao Li,  Yanci Nie and Xi Zhang}
\address{Chao Li\\School of Mathematical Sciences\\
University of Science and Technology of China\\
Hefei, 230026,P.R. China\\}\email{leecryst@mail.ustc.edu.cn}

\address{Yanci Nie\\School of Mathematical Sciences\\
Xiamen University\\
Xiamen, 361005\\ } \email{nieyanci@xmu.edu.cn}
\address{Xi Zhang\\School of Mathematical Sciences\\
University of Science and Technology of China\\
Hefei, 230026,P.R. China\\ } \email{mathzx@ustc.edu.cn}
\thanks{The authors were supported in part by NSF in China,  No.11625106, 11571332 and 11721101.}
\maketitle

\begin{abstract}
In this paper, we study numerically flat holomorphic vector bundles over a compact non-K\"ahler manifold $(X, \omega)$ with the Hermitian metric $\omega$ satisfying the Gauduchon and Astheno-K\"ahler conditions. We prove that
 numerically flatness is equivalent to numerically effectiveness with vanishing first Chern number, semistablity with vanishing first and second Chern numbers, approximate Hermitian flatness and the existence of a filtration whose quotients are Hermitian flat. This gives an affirmative answer to the question proposed by Demailly, Peternell and Schneider.
\end{abstract}

\section{introduction}
The notion of positivity plays an important role in algebraic geometry and complex geometry. Let $L$ be a line bundle over a compact complex manifold $X$. $L$ is said to be positive (semipositive) if there is a Hermitian metric $h$ on $L$ such that the curvature $\sqrt{-1}\Theta(L, h)>0$$(\geq 0)$. A natural generalization and more flexible notion is numerically effective (nef for short). When $X$ is projective, $L$ is said to be nef if $L\cdot C\geq 0$ for every compact curve $C\subset X.$ However, when $X$ is just a general compact complex manifold, there maybe no compact curves over $X$. Motived by the following property of nef line bundles over projective manifolds,
\begin{lemma}\mbox{\rm(\cite[Lemma 1.1]{DPS})}
Let $A$ be an ample line bundle over a projective manifold $X$. Then a line bundle $L$ is nef if and only if $L^k\otimes A$ is ample for every integer $k\geq 0.$
\end{lemma}
Demailly, Peternell and Schneider (\cite{DPS}) generalized this definition to general compact complex manifolds in terms  of curvature, that is
\begin{definition}
 Let $(X,\omega)$ be an $n$-dimensional compact Hermitian manifold.  A line bundle $L$ over  $(X,\omega)$ is  said to be numerically effective, if for every $\epsilon>0,$ there exists a smooth metric $h_{\epsilon}$ on $L$ such that  the curvature $\sqrt{-1}\Theta(L, h_{\epsilon})=\sqrt{-1}\bar{\partial}\partial\log h_{\epsilon}\geq-\epsilon\omega$.
 \end{definition}
 This means the curvature of $L$ can have arbitrary small negative part. It is obvious that a Hermitian flat line bundle is nef.  A vector bundle $E$ of rank $r\geq 2$ is said to be nef if the anti tautological line  bundle $\mathcal{O}_E(1)$ on the projective bundle $PE$ is nef. $E$ is said to be numerically flat (nflat for short), if both $E$ and its dual $E^*$ are nef.

In \cite{DPS},  the authors established the relationship between Hermitian flatness and nflatness. For line bundles, nflatness is equivalent to Hermitian flatness(\cite[Corollary 1.5]{DPS}). As to vector bundles of higher rank, they showed that a holomorphic vector bundle over a compact K\"ahler manifold  is nflat if and only if it admits a filtration by sub-bundles such that the quotients are Hermitian flat. Based on the above results, they raised an interesting question whether the above result holds in non-K\"ahler case and pointed out the difficulty is to show the second Chern number of a numerically flat vector bundle is zero. In \cite{DPS}, they obtained it by the Fulton-Lazarsfeld inequalities for Chern classes of nef vector bundles (\cite[Theorem 2.5]{DPS}) which only hold over compact K\"ahler manifolds. Under the assumption of $c_2(E)=0$, Biswas and Pingali (\cite{BP}) obtained a characterization of numerically flat bundle $E$ on a compact complex  manifold $(X,\omega)$ with $\omega$ satisfying Gauduchon ($\partial\bar{\partial}\omega^{n-1}=0$) and Astheno-K\"ahler ($\partial\bar{\partial}\omega^{n-2}=0$) conditions which make the first and second Chern numbers well-defined (\cite[Theorem 3.2]{BP}). 

 In this paper, we consider nflat vector bundles over  compact non-K\"ahler manifolds without the assumption of $c_2(E)=0$. In fact, we obtain the equivalence between nflatness, nefness with vanishing first Chern number, semistability with vanishing first and second Chern numbers, approximate Hermitian flatness and the existence of the filtration by sub-bundles whose quotients are Hermitian flat. That is

\begin{theorem}\label{thm2}
Let $(X,\omega)$ be a compact Hermitian manifold of dimension $n$ with $\omega$ satisfying $\partial\bar{\partial}\omega^{n-1}=\partial\bar{\partial}\omega^{n-2}=0.$ Let $(E,\bar{\partial}_E)$ be a holomorphic vector bundle over $X$. Then the following statements are equivalent:
\begin{enumerate}
\item  [\rm{(1)}] Numerically flat;
\item  [\rm{(2)}] Numerically effective with $ch_1(E)\cdot [\omega^{n-1}]=0;$
\item  [\rm{(3)}] Semistable with $ch_1(E)\cdot [\omega^{n-1}]=ch_2(E)\cdot [\omega^{n-2}]=0$;
\item [\rm{(4)}] Approximately Hermitian flat;
\item [\rm{(5)}] There exists a filtration
$$0\subset E_0\subset E_1\cdots E_l=E$$
 by sub-bundles  whose quotients are Hermitian flat.
\end{enumerate}

\end{theorem}

\begin{remark}
It has been proved by Gauduchon (\cite{Gaud}) that if $X$ is compact, then there exists a Gauduchon metric  in the conformal class of every Hermitian metric. So any compact complex manifold $X$ admits a Hermitian metric $\omega $ satisfying  Gauduchon condition. If in addition $X$ is a complex surface, i.e. $\dim X=2$, then $\omega$ automatically also satisfies Astheno-K\"ahler condition.
By this fact,  it is easy to check that, if $(X_{1}, \omega_{1})$ is  a K\"ahler manifold and $(X_{2}, \omega_{2})$ is a Gauduchon surface, then $(X, \omega)=(X_{1}\times X_{2}, \omega_{1}+\omega_{2})$ also satisfies Gauduchon and Astheno-K\"ahler conditions. For another examples of Gauduchon Astheno-K\"ahler manifolds, see \cite{LY87,LYZ,FGV,LU}.
for instance.
\end{remark}

\begin{corollary}
Let $X$ be  a compact complex surface, and $(E,\bar{\partial}_E)$ be a holomorphic vector bundle over $X$. Then the following statements are equivalent:
\begin{enumerate}
\item  [\rm{(1)}] Numerically flat;
\item [\rm{(2)}] Approximately Hermitian flat;
\item [\rm{(3)}] There exists a filtration
$$0\subset E_0\subset E_1\cdots E_l=E$$
 by sub-bundles  whose quotients are Hermitian flat.
\end{enumerate}
\end{corollary}

We give an overview of the proof.
%By the property of nefness, it is easy to  get that the existence of a filtration by sub-bundles whose quotients are Hermitian flat can derive nflatness.
The key points are to show a numerically effective holomorphic bundle with vanishing first Chern number is semi-stable with vanishing first and second Chern numbers and a semistable holomorphic bundle with vanishing first and second Chern numbers is approximate Hermitian flat.

As to the former,  first following the argument of Step $1$ in the proof of \cite[Theorem 1.18]{DPS} and Lemma \ref{nefline}, we have nef vector bundles with vanishing first Chern number are semi-stable and the determinant line bundles are Hermitian flat. Then using the push forward formula of Segre forms by the first Chern form of the anti tautological line bundle $\mathcal{O}_E(1)$ on $PE$ and Bogomolov type inequality (Proposition \ref{Bogomolov}), we obtain that the second Chern number of nflat vector bundles is zero.

And as to the later,  we have the following theorem:
\begin{theorem}\label{thm1}
Let $(X,\omega)$ be an $n$-dimensional compact Hermitian manifold with $\omega$ satisfying $\partial\bar{\partial}\omega^{n-1}=\partial\bar{\partial}\omega^{n-2}=0.$ Let $(E,\bar{\partial}_E)$ be a holomorphic vector bundle of rank $r$. If $(E,\bar{\partial}_E)$ is semistable with $ch_1(E)\cdot [\omega^{n-1}]=ch_2(E)\cdot [\omega^{n-2}]=0,$ then it is approximate Hermitian flat.
\end{theorem}
In \cite{NZ1}, by using the Hermitian-Yang-Mills flow,  we derived that a semistable holomorphic vector bundle with vanishing first and second Chern numbers over a compact K\"ahler manifold is approximate Hermitian flat.  Let $(E,\bar{\partial}_E)$ be a holomorphic vector bundle over compact complex manifold $(X,\omega)$ and $H$ be an arbitrary Hermitian metric on $E$. When $\omega$ is K\"ahler, by the Chern-Weil theory, we have

\begin{equation}\label{equ6}
\begin{split}
\int_X |F_{H,\bar{\partial}_E}|^2_{H}\frac{\omega^n}{n!}
=&\int_X|\sqrt{-1}\Lambda_{\omega}F_{H,\bar{\partial}_E}-\lambda\cdot \Id_E|^2_{H}\frac{\omega^n}{n!}-8\pi^2\int_X ch_2(E,H)\wedge \frac{\omega^{n-2}}{(n-2)!}\\
&+\lambda^2 \mbox{\rm rank}(E)\Vol(X,\omega),
\end{split}
\end{equation}
where $\lambda=\displaystyle{\frac{2\pi \mu_{\omega}(E)}{\Vol(X,\omega)}}.$  In fact, (\ref{equ6}) also holds when $\omega$ satisfies $\partial\bar{\partial}\omega^{n-1}=\partial\bar{\partial}\omega^{n-2}=0$. And  when $ch_1(E,H)\cdot[\omega^{n-1}]=ch_2(E,H)\cdot [\omega^{n-2}]=0$, we have
\begin{equation}\label{equ9}
\int_X |F_{H,\bar{\partial}_E}|^2_H=\int_X|\sqrt{-1}\Lambda_{\omega}F_{H,\bar{\partial}_E}|^2_H.
\end{equation}
 Since $\deg_{\omega}(E)=0,$  consider the following Hermitian-Yang-Mills flow
 \begin{equation}\label{HYM}
\begin{cases}
H(t)^{-1}\frac{\partial H}{\partial t}=-2\sqrt{-1}\Lambda_{\omega}F_{H,\bar{\partial}_E},\\
H(0)=H_0,
\end{cases}
\end{equation}
where $H_0$ is an arbitrary Hermitian metric. It has been proved (\cite[Equation 3.14]{LZ}) that when $(E,\bar{\partial}_E)$ is semi-stable,  along the flow (\ref{HYM}) , it holds
\begin{equation}\label{equ5}
\int_X|F_{H,\bar{\partial}_E}|^2_{H(t)}=\int_X|\sqrt{-1}\Lambda_{\omega}F_{H(t),\bar{\partial}_E}|^2_{H(t)}\rightarrow 0,\ \ \ \ \ \ \mbox{\rm as}\ \ t\rightarrow \infty.
\end{equation}
Then by the small energy regularity theorem of Yang-Mills flow, we (\cite[Equation 3.3]{NZ1}) obtained that
  $$\sup_X|F_{H(t),\bar{\partial}_E}|^2_{H(t)}\rightarrow 0,\ \ \ \ \ \mbox{\rm as}\ \ t\rightarrow \infty.$$

  However, when $\omega$ is not K\"ahler, we do not know whether (\ref{equ5}) still holds  along the Hermitian-Yang-Mills flow and can not generalize the argument in \cite{NZ1} directly.  In this paper, we combine the continuity method and the Hermitian-Yang-Mills flow to construct the approximate Hermitian flat structure. Fixing a proper Hermitian metric $K$ with $\mbox{tr}(\sqrt{-1}\Lambda_{\omega}F_{K,\bar{\partial}_E}-\lambda\cdot \Id_E)=0$, we consider the following perturbed equation
\begin{equation}\label{equ8}
\sqrt{-1}\Lambda_{\omega}F_{H_{\varepsilon},\bar{\partial}_E}-\lambda \cdot\Id_E+\varepsilon\log f_{\varepsilon}=0,\ \ \ \ \varepsilon\in (0,1],
\end{equation}
 where $f_{\varepsilon}=K^{-1}H_{\varepsilon}\in \Herm^+(E, K).$ It has been proved that (\ref{equ8}) is solvable for $\varepsilon\in (0,1]$.  Then for each $\varepsilon\in(0,1]$, consider the Hermitian-Yang-Mills flow with $H_{\varepsilon}$ as the initial metric
 \begin{equation*}
  \begin{cases}
  H^{-1}_{\varepsilon}(t)\frac{\partial H_{\varepsilon}(t)}{\partial t}=-2\sqrt{-1}\Lambda_{\omega}F_{H_{\varepsilon}(t),\bar{\partial}_E},\\
  H_{\varepsilon}(0)=H_{\varepsilon}.
  \end{cases}
  \end{equation*}

From \cite[Theorem 3.2]{NZ2}, we have that when $(E,\bar{\partial}_E)$ is semi-stable, it holds
$$\sup_X|\Lambda_{\omega}F_{H_{\varepsilon},\bar{\partial}_E}|_{H_{\varepsilon}}\rightarrow 0,\ \ \ \ \ \mbox{as}\ \ \varepsilon\rightarrow 0.$$
By Equality (\ref{equ9}) and Lemma \ref{lmm1}, it holds
\begin{align*}
\int_X|F_{A_{\varepsilon}(t)}|_{H_{\varepsilon}}^2\frac{\omega^n}{n!}
&\leq \int_X |F_{A_{\varepsilon}(0)}|^2_{H_{\varepsilon}}\frac{\omega^n}{n!}= \int_X |F_{H_{\varepsilon},\bar{\partial}_E}|^2_{H_{\varepsilon}}\frac{\omega^n}{n!}\\
&= \int_X|\sqrt{-1}\Lambda_{\omega}F_{H_{\varepsilon}, \bar{\partial}_E}|^2_{H_{\varepsilon}}\frac{\omega^n}{n!}\rightarrow 0,\ \ \varepsilon\rightarrow 0,
\end{align*}
where $A_{\varepsilon}(t)$ is the solution of (\ref{equ7}). This implies that when $\varepsilon$ is small enough, the $L^2$-norm of the curvature $F_{A_{\epsilon}(t)}$ is also small. Then we use the small energy regularity to obtain the approximate Hermitian flat structure.

This paper is organized as below. In Section $2$, we will recall some basic notions and related properties. In Section $3$, we give a detailed proof of Theorem \ref{thm1}.  In Section $4$, we give a detailed proof of Theorem \ref{thm2}.

\section{Preliminary}
In this section, we recall some definitions and properties of nef vector bundles needed in this paper.
\subsection{Some definitions}
Let $X$ be an $n$-dimensional compact complex manifold and $g$ be a Hermitian metric with associated $(1,1)$-form $\omega$. $g$ is called Gauduchon if $\omega$ satisfies $\partial\bar{\partial}\omega^{n-1}=0$.  If $\partial\bar{\partial}\omega^{n-2}=0$, the Hermitian metric $g$  is said to be Astheno-K\"ahler  which was introduced by Jost and Yau in \cite{JY93}. In this paper, we assume  $\omega$ satisfies $\partial\bar{\partial}\omega^{n-1}=\partial\bar{\partial}\omega^{n-2}=0.$

 Let $(L, h)$  be a Hermitian line bundle over $X$.  The $\omega$-degree of $L$ is defined by
\begin{equation*}
\deg_{\omega}(L):=\int_X c_1(L, A_h)\wedge\displaystyle{\frac{\omega^{n-1}}{(n-1)!}},
\end{equation*}
where $c_1(L,A_h)$ is the first Chern form of $L$ associated with the Chern connection  $A_h$ with respect to the Hermitian metric $h$.  Since $\partial\bar{\partial}\omega^{n-1}=0,$  $\deg_{\omega}(L)$ is  well defined and independent of the choice of  metric $h$ (\cite[p.~34-35]{MA}).

Now given a coherent analytic sheaf  $\mathcal{F}$ of rank $s$, we consider the  determinant line bundle  $\det{\mathcal{F}}=(\wedge^s\mathcal{F} )^{**}.$  Define the $\omega$-degree of $\mathcal{F}$ by
\begin{equation*}
\deg_{\omega}(\mathcal{F}):=\mbox{deg}_{\omega}(\det{\mathcal{F}}).
\end{equation*}
If $\mathcal{F}$ is non-trivial and torsion free, the $\omega$-slope of $\mathcal{F}$ is defined by
  $$\mu_{\omega}(\mathcal{F})=\frac{\deg_{\omega}(\mathcal{F})}{\mbox{rank}(\mathcal{F})}.$$

Let $(E,\bar{\partial}_E,H)$ be a rank $r$ holomorphic Hermitian vector bundle. Denote $D_{H,\bar{\partial}_E}$ the Chern connection of $(E, \bar{\partial}_E, H)$ and $F_{H, \bar{\partial}_E}=D^2_{H,\bar{\partial}_E}$ the Chern curvature. Then the corresponding Chern forms $c_k(E,H)\in A^{k,k}(X)$ are computed by
\begin{equation}
\det \left(\mbox{Id}_E+\frac{\sqrt{-1}t}{2\pi}F_{H,\bar{\partial}_E}\right)=\sum_{i=0}^{\min{\{r,n\}}} c_i(E,H)t^i.
\end{equation}
Let $H_1$ and $H_2$ be two Hermitian metrics on $E$. It has been proved by Donaldson \cite[Proposition 6]{Donaldson85} that for every $1\leq k\leq \min{\{r,n\}}$, there exists $R_{k-1}(H_1,H_2)\in A^{k-1,k-1}(X)$ such that
\begin{equation}
c_k(E,H_1)-c_k(E,H_2)=\sqrt{-1}\bar{\partial}\partial R_{k-1}(H_1, H_2).
\end{equation}

So when $\partial\bar{\partial}\omega^{n-1}=\partial\bar{\partial}\omega^{n-2}=0,$
\begin{equation}c_1(E)\cdot[\omega^{n-1}]=\int_X c_1(E, H)\wedge \omega^{n-1}\end{equation}
and
\begin{equation}
c_2(E)\cdot [\omega^{n-2}]=\int_X c_2(E,H)\wedge \omega^{n-2},\ \ \ \ \ c_1^2(E)\cdot[\omega^{n-2}]=\int_X c^2_1(E,H)\wedge \omega^{n-2}
\end{equation}
are well-defined and independent of the Hermitian metrics on $E$, where
$$[\omega^{n-1}]\in H_A^{n-1,n-1}(X),\ \ \ [\omega^{n-2}]\in H_A^{n-2,n-2}(X)$$
and
 $$c_1(E)\in H_{BC}^{1,1}(X),\ \ \ c_2(E), c_1^2(E)\in H_{BC}^{2,2}(X).$$

\begin{remark}
The Bott-Chern cohomology and Aeppli cohomology are defined by
\begin{equation*}
H^{\bullet,\bullet}_{BC}(X)=\frac{\mbox{\rm Ker} \partial \cap \mbox{\rm Ker} \bar{\partial}}{\mbox{\rm Imm} \partial\bar{\partial}}
\end{equation*}
and
\begin{equation*}
H_A^{\bullet,\bullet}(X)=\frac{\mbox{\rm Ker}\partial\bar{\partial}}{\mbox{\rm Imm }\partial\cap \mbox{\rm Imm} \bar{\partial}}.
\end{equation*}
\end{remark}

\begin{definition}
Let $(E, \bar{\partial}_E)$ be a holomorphic vector bundle over $X$. We say $E$ is  $\omega$-stable ($\omega$-semi-stable) in the sense of Mumford-Takemoto  if for every proper coherent sub-sheaf $ \mathcal{F}\hookrightarrow E$, it holds
\begin{equation*}
\mu_{\omega}(\mathcal{F})<\mu_{\omega}(E)(\mu_{\omega}(\mathcal{F})\leq\mu_{\omega}(E)).
\end{equation*}
\end{definition}

\begin{definition}
A Hermitian metric $H$ on $E$ is said to be $\omega$-Hermitian-Einstein if the Chern curvature $F_{H,\bar{\partial}_E}$  satisfies the Einstein condition
 $$ \sqrt{-1}\Lambda_{\omega}F_{H,\bar{\partial}_E}=\lambda\cdot \mbox{\rm Id}_{E},$$
 where $\lambda=\displaystyle{\frac{2\pi \mu_{\omega}(E)}{\Vol(X,\omega)}}.$
 \end{definition}

 \begin{remark}By \cite[Proposition 1.5 or Lemma 2.1]{Bru}, when checking the stability of a holomorphic vector bundle, we only need to consider proper saturated sub-sheaves, i.e. sub-sheaves with torsion free quotients.
 \end{remark}

 The classic Donaldson-Uhlenbeck-Yau theorem (\cite{NS,Donaldson85,UY,Donaldson87}) tells us there exist Hermitian-Einstein metrics on holomorphic vector bundles over compact K\"ahler manifolds if they are stable and was generalized by Li and Yau (\cite{LY87}) for general compact Gauduchon manifolds.
 When the K\"ahler form is understood, we omit the subscript $\omega$ in the above definitions.

 \begin{definition}
 A holomorphic vector bundle $(E,\bar{\partial}_E)$ is said to admit an approximate Hermitian-Einstein structure, if for every $\epsilon>0,$ there exists a Hermitian metric $H_{\epsilon}$, such that
 \begin{equation}
 \sup_X |\sqrt{-1}\Lambda_{\omega}F_{H_{\epsilon},\bar{\partial}_E}-\lambda\cdot \mbox{\rm Id}_{E}|_{H_{\epsilon}}<\epsilon.
 \end{equation}
\end{definition}
Kobayashi (\cite{Kobayashi}) introduced this notion for a holomorphic vector bundle. Similar with the relationship between stability and the existence of Hermitian-Einstein metrics, a holomorphic vector bundle admits an approximate Hermitian-Einstein structure if it is semistable. It was  proved by Kobayashi (\cite{Kobayashi}) for projective manifolds, by Jacob (\cite{Ja}), Li and  Zhang (\cite{LZ}) for compact K\"ahler manifolds and by Nie and Zhang (\cite{NZ2}) for general compact Gauduchon manifolds.
Furthermore, if $\omega$ is both Gauduchon and Astheno-K\"ahler, we have the following Bogomolov type inequality, which was first obtained by Bogomolov (\cite{Bogomolov}) for semi-stable holomorphic vector bundles on complex algebraic surfaces.
 \begin{proposition}\label{Bogomolov}
Let $(X,\omega)$ be an $n$-dimensional compact complex manifold with $\omega$ satisfying $\partial\bar{\partial}\omega^{n-1}=\partial\bar{\partial}\omega^{n-2}=0$ and $(E,\bar{\partial}_E)$ be a rank $r$ holomorphic vector bundle. If $E$ is semi-stable, then we have the  following Bogomolov type inequality
\begin{equation}
4\pi^2\left( 2c_2(E)-\frac{r-1}{r}c_1^2(E)\right)\cdot [\omega^{n-2}]\geq 0.
\end{equation}
\end{proposition}

\begin{proof}
From the above,  when $\omega$ satisfies $\partial\bar{\partial}\omega^{n-1}=\partial\bar{\partial}\omega^{n-2}=0$, we have $$4\pi^2\left( 2c_2(E)-\frac{r-1}{r}c_1^2(E)\right)\cdot [\omega^{n-2}]$$
 is well-defined and independent of the choice of the Hermitian metrics on $(E,\bar{\partial}_E)$. Endowed $E$ with an arbitrary Hermitian metric $H$, we have
 \begin{equation}\label{equ15}
 \begin{split}
 &4\pi^2\left( 2c_2(E)-\frac{r-1}{r}c_1^2(E)\right)\cdot \frac{[\omega^{n-2}]}{(n-2)!}\\
=&4\pi^2\int_X \left( 2c_2(E,H)-\frac{r-1}{r}c_1(E,H)\wedge c_1(E,H)\right) \wedge \frac{\omega^{n-2}}{(n-2)!}\\
=&\int_X \mbox{\rm tr}(F_{H,\bar{\partial}_E}^{\perp}\wedge F_{H,\bar{\partial}_E}^{\perp}) \wedge \frac{\omega^{n-2}}{(n-2)!}\\
=&\int_X\left(|F_{H,\bar{\partial}_E}^{\perp}|_H^2-|\Lambda_{\omega} F_{H,\bar{\partial}_E}^{\perp}|_H^2\right) \frac{\omega^n}{n!}\\
\geq &-\int_X|\sqrt{-1}\Lambda_{\omega}F_{H,\bar{\partial}_E}-\lambda\cdot \Id_E-\frac{1}{r}\mbox{\rm tr}(\sqrt{-1}\Lambda_{\omega}F_{H,\bar{\partial}_E}-\lambda\cdot\Id_E)\Id_E|^2_H\frac{\omega^n}{n!},
 \end{split}
 \end{equation}
where $F_{H,\bar{\partial}_E}^{\perp}=F_{H,\bar{\partial}_E}-\frac{ \mbox{\rm tr} F_{H,\bar{\partial}_E}}{r}\mbox{\rm Id}_E$ is the trace free part of $F_{H,\bar{\partial}_E}.$

Since $(E,\bar{\partial}_E)$ is semi-stable, $(E,\bar{\partial}_E)$ admits an approximate Hermitian-Einstein structure (\cite{NZ2}), that is for every $\epsilon>0$, there exists $H_{\epsilon}$ such that
 $$ \sup_X |\sqrt{-1}\Lambda_{\omega}F_{H_{\epsilon},\bar{\partial}_E}-\lambda\cdot \mbox{\rm Id}_{E}|_{H_{\epsilon}}<\epsilon.$$
 Then
\begin{equation}\label{equ16}
\int_X|\sqrt{-1}\Lambda_{\omega}F_{H_{\epsilon}, \bar{\partial}_E}-\lambda \cdot\Id_E-\frac{1}{r}\mbox{\rm tr}(\sqrt{-1}\Lambda_{\omega}F_{H_{\epsilon}, \bar{\partial}_E}-\lambda \cdot\Id_E)\mbox{\rm Id}_E|^2_{H_{\epsilon}}\frac{\omega^n}{n!}\rightarrow 0,\ \ \ \ \epsilon \rightarrow 0.
\end{equation}
Therefore, by Equation (\ref{equ15}) and (\ref{equ16}), we have
\begin{equation*}
\begin{split}
4\pi^2\left( 2c_2(E)-\frac{r-1}{r}c_1^2(E)\right)\cdot \frac{[\omega^{n-2}]}{(n-2)!}
\geq 0.
\end{split}
\end{equation*}

\end{proof}

\begin{definition}
A holomorphic vector bundle $(E,\bar{\partial}_E)$ is said to be approximate Hermitian flat, if for every $\epsilon>0,$ there exists a Hermitian metric $H_{\epsilon}$ such that
  $$\sup_X |F_{H_{\epsilon}, \bar{\partial}_E}|_{H_{\epsilon}}<\epsilon.$$

\end{definition}

%\begin{definition} (nef line bundle) \label{def1}
%A line bundle $L$ over $X$ is  said to be numerically effective ({\it {nef}} for short),  if for every $\epsilon>0,$ there exists a smooth metric $h_{\epsilon}$ on $L$ such that $\Theta(L)_{h_{\epsilon}}>-\epsilon\omega.$
%$L$ is said to be numerically flat if $L$ and its dual $L^*$ are both nef.
%\end{definition}
\subsection{Basic properties of nef vector bundles}

In this subsection, we will present some basic properties of nef vector bundles. For the detailed proof, please see reference \cite{DPS}.

\begin{proposition}\mbox{\rm(\cite[Corollary 1.5]{DPS})}\label{prop3}
$L$ is numerically flat if and only if it is Hermitian flat.
\end{proposition}

\begin{lemma}\label{nefline}
Let  $(X,\omega)$ be an $n$-dimensional compact Hermitian manifold with $\omega$ satisfying $\partial\bar{\partial}\omega^{n-1}=\partial\bar{\partial}\omega^{n-2}=0$. Let $L$ be a holomorphic line bundle on $X$. If $L$ is nef and $c_1(L)\cdot [\omega^{n-1}]=0$, then $L$ is Hermitian flat.
\end{lemma}

\begin{proof}
From the above, we have when $\partial\bar{\partial}\omega^{n-1}=\partial\bar{\partial}\omega^{n-2}=0$, $c_1(L)\cdot [\omega^{n-1}]$ and
$c_1(L)^2\cdot [\omega^{n-2}]$ are well-defined and can be computed by the Chern form $\Theta(L, h)$ of an arbitrary Hermitian metric $h$ on $L$. On one hand, since $L$ is nef,  for every $\epsilon>0,$ there exists a Hermitian metric $h_{\epsilon}$ such that $\sqrt{-1}\Theta(L, h_{\epsilon})\geq -\epsilon\omega.$

So
\begin{align*}
&0\leq \int_X (\frac{\sqrt{-1}}{2\pi}\Theta(L, h_{\epsilon})+\epsilon\omega)^2\wedge \omega^{n-2}\\
&=\int_X (\frac{\sqrt{-1}}{2\pi}\Theta(L, h_{\epsilon}))^2\wedge \omega^{n-2}+2\epsilon\int_X\frac{\sqrt{-1}}{2\pi}\Theta(L, h_{\epsilon})\wedge \omega^{n-1}+\epsilon^2\int_X \omega^n\\
&\rightarrow c_1(L)^2\cdot[\omega^{n-2}],\ \ \ \ \epsilon\rightarrow 0.
\end{align*}
This implies
\begin{equation}\label{c1}
c_1(L)^2\cdot [\omega^{n-2}]\geq 0.
\end{equation}
And on the other hand, since $c_1(L)\cdot [\omega^{n-1}]=0$, we can find a Hermitian metric $h$ on $E$ such that

\begin{equation}\sqrt{-1}\Lambda_{\omega}\Theta(L)_h=0.
\end{equation}
From (\ref{equ6}), we have
\begin{equation}\label{c2}
\begin{split}
c_1(L)^2\cdot [\omega^{n-2}/(n-2)!]
&=2ch_2(L)\cdot  [\omega^{n-2}/(n-2)!]\\
&=\frac{1}{4\pi^2}\left(\int_X |\sqrt{-1}\Lambda_{\omega}\Theta(L,h)|^2-\int_X|\Theta(L,h)|^2\right)\\
&=-\frac{1}{4\pi^2}\int_X|\Theta(L, h)|^2.
\end{split}
\end{equation}

Combining (\ref{c1}) and (\ref{c2}), we have  $\Theta(L, h)=0$. This concludes the proof.
\end{proof}

Let $m$ be a positive integer and let $S^m E$ be the $m$-th symmetric power of $E$, then $\det S^mE=(\det E)^{N}$ where \begin{equation*}N=\frac{m\rank (S^mE)}{\rank (E)}.\end{equation*}
Together with \cite[Theorem 1.12]{DPS}, we can easily check that

\begin{proposition}\label{detnef}
Let $(E,\bar{\partial}_E)$ be a holomorphic vector bundle over $(X,\omega)$. If $E$ is nef, then $\det E$ is nef.
\end{proposition}

By Proposition \ref{detnef} and  Proposition \ref{prop3}, we have

\begin{proposition}\label{prop5}
Let $(E,\bar{\partial}_E)$ be a holomorphic vector bundle over $(X,\omega)$. If $E$ is nflat, then $\det E$ is Hermitian flat.
%\begin{enumerate}\label{prop5}
%\item $E$ is nflat if and only if both $\mathcal{O}_E(1)$ and $\mathcal O_{E^*}(1)$ are nef.
%\item $E$ is nflat, then $\det E$ is Hermitian flat.
%\end{enumerate}
\end{proposition}

\begin{proposition}\mbox{\rm (\cite[Proposition 1.14]{DPS})}\label{prop4}
Let $E$ and $F$ be two holomorphic vector bundles over $X$. If $E$ and $F$ are nef, then $E\otimes F$ is nef.
\end{proposition}

\begin{proposition}\mbox{\rm (\cite[Proposition 1.15]{DPS})}\label{prop2}
Let $0\rightarrow F\rightarrow E\rightarrow Q\rightarrow 0$ be an exact sequence of holomorphic vector bundles. Then
        \begin{enumerate}
         \item If $E$ is nef, then $Q$ is nef;
       \item If $F$ and $Q$ are nef, then $E$ is nef;
       \item If $E$ and $(\det Q)^{-1}$ are nef, then $F$ is nef.
      \end{enumerate}
\end{proposition}

%\begin{proposition}\label{prop1}
%Let $0\rightarrow S\rightarrow E \rightarrow Q\rightarrow 0$ be an exact sequence of holomorphic vector bundles over $X$. If $S$ and $Q$ are approximate Hermitian flat, then $E$ is approximate Hermitian flat.
%\end{proposition}

%\begin{proposition}
%If $E$ is approximate Hermitian flat, then $E$ is nflat.
%\end{proposition}

\subsection{ Segre forms  } In this subsection, we will introduce the push forward formula of Segre forms which was proved by Guler(\cite{Guler}) for projective manifolds and by Diverio(\cite{Diverio}) for general compact complex manifolds.

Let $(E,\bar{\partial}_E)\rightarrow X$ be a rank $r$ holomorphic vector bundle and $c_{\bullet}(E)=1+c_1(E)+\cdots+c_r(E)\in H^{\bullet}(X,\mathbb{Z})$ the total Chern class of $E$. The inverse of $c_{\bullet}(E)$ is by definition the total Segre class $s_{\bullet}(E)=1+s_1(E)+\cdots+s_n(E)\in H^{\bullet}(X,\mathbb{Z})$.
%\begin{equation}
%\begin{split}
%&s_{1}(E)=-c_1(E),\\
%&s_2(E)=c_1(E)^2-c_2(E),\\
%&\cdots
%\end{split}
%\end{equation}
Endow $(E,\bar{\partial}_E)$ with a Hermitian metric $H$. Then from the Chern-Weil theory, the Segre forms $s_k(E,H)$ can be defined inductively by the relation
\begin{equation}
s_k(E,H)+c_1(E,H)s_{k-1}(E,H)+\cdots+c_k(E,H)=0,\ \ \ \ 0\leq k\leq n.
\end{equation}
For example,

\begin{align}
&s_{1}(E,H)=-c_1(E,H),\\
&\label{segre2} s_2(E,H)=c_1(E,H)^2-c_2(E,H).
\end{align}

{\bf Push forward of forms.}   Let $M,\, N$ be oriented differential manifolds of dimension $m,\, n$ ($m>n$) and $f: M\rightarrow N$ be a proper submersion. Set $s=m-n$. Then for any smooth $(p+s)$-form $\eta$ on $M$, there exists a unique smooth $p$-form $\xi$ on $N$ such that the equality
\begin{equation}
\int_M\eta\wedge f^*\phi=\int_N\xi\wedge\phi
\end{equation}
holds for any smooth $(n-p)$-form $\phi$ on $N$ with compact support.

 Given a Hermitian metric $H$ on $E$, denote $h$ the induced  metric on $\mathcal{O}_E(1)\rightarrow PE$ and $ \Xi=\frac{\sqrt{-1}}{2\pi}\Theta(\mathcal{O}_E(1), h).$
Then we have  the following push forward formula of Segre forms:
\begin{lemma}\mbox{\rm (\cite[Proposition 1.1]{Diverio})}\label{lmm5}
For each $k=0,\cdots, n,$ the equality
\begin{equation}
\pi_{\star}(\Xi^{r-1+k})=s_k(E,H),
\end{equation}
holds, where $s_0(E, H)$ is the function on $X$ and constantly equal to $1$.
\end{lemma}

%How to define the inverse of a class.
%--------------------------------------------------------------------------------------------------------------------------------------------------------------------------------------------------------------------------
\section{Proof of Theorem \ref{thm1}}

In this section, we will combine the continuity method and the Hermitian-Yang-Mills flow to construct the approximate Hermitian flat structure. First, we introduce the continuity method, the Hermitian-Yang-Mills flow and some related properties.

Let $(E,\bar{\partial}_E)$ be a holomorphic bundle over a compact Gauduchon manifold $(X,\omega)$.
Fix a proper Hermitian metric $K$ with $\mbox{tr}(\sqrt{-1}\Lambda_{\omega}F_{K,\bar{\partial}_E}-\lambda\cdot \mbox{Id}_E)=0$. Consider the perturbed equation
\begin{equation}\label{equ12}
\sqrt{-1}\Lambda_{\omega}F_{H_{\varepsilon},\bar{\partial}_E}-\lambda \cdot\Id_E+\varepsilon\log f_{\varepsilon}=0,\ \ \ \ \varepsilon\in (0,1],
\end{equation}
where $f_{\varepsilon}=K^{-1}\cdot H_{\varepsilon}$ and $\lambda=\frac{2\pi\mu_{\omega}(E)}{\Vol(X,\omega)}$.  When $(E, \bar{\partial}_E)$ is semi-stable, we studied the asymptotic properties as $\varepsilon\rightarrow 0$. In fact, we proved,

\begin{lemma}\mbox{\rm (\cite[Theorem 3.2]{NZ2})}\label{lmm3}
If $(E,\bar{\partial}_E)$ is semi-stable, then
\begin{equation}
\sup_X |\sqrt{-1}\Lambda_{\omega}F_{H_{\varepsilon},\bar{\partial}_E}-\lambda\cdot \Id_E|_{H_{\varepsilon}} \rightarrow 0,\ \ \ \ \varepsilon\rightarrow 0.
\end{equation}
\end{lemma}

Given an arbitrary metric $H_0$ on $(E, \bar{\partial}_E)$, consider the Hermitian-Yang-Mills flow,
\begin{equation}\label{equ1}
\begin{cases}
H(t)^{-1}\frac{\partial H(t)}{\partial t}=-2(\sqrt{-1}\Lambda_{\omega}F_{H(t),\bar{\partial}_E}-\lambda \cdot \Id_E),\\
H(0)=H_0.
\end{cases}
\end{equation}
 Denote the space of connections of $E$ compatible with $H_{0}$ by
$\mathcal{A}_{H_0}$, the space of unitary integrable connections of $E$  by $\mathcal{A}_{H_0}^{1,1}$ and  the complex gauge group (resp. unitary gauge group) of  $(E, H_{0})$ by $\mathcal{G}^{\mathbb{C}}$ (resp. $\mathcal{G}$, where $\mathcal{G}=\{\sigma \in \mathcal{G}^{\mathbb{C}}| \sigma^{\ast H_0}\sigma=\mathrm{Id}\}$). $\mathcal{G}^{\mathbb{C}}$ acts on the space $\mathcal{A}_{H_0}$ as follows: for
$\sigma \in \mathcal{G}^{\mathbb{C}}$ and $A\in \mathcal{A}_{H_0}$,

\begin{equation}\label{id2}
\overline{\partial }_{\sigma(A)}=\sigma \circ \overline{\partial
}_{A}\circ \sigma^{-1}, \quad \partial _{\sigma (A)}=(\sigma^{\ast
H_{0}})^{-1} \circ \partial _{A}\circ \sigma^{\ast H_{0}}.
\end{equation}

From \cite{NZ3}, we have the heat flow (\ref{equ1}) is equivalent to the following flow
\begin{equation}\label{equ2}
\begin{cases}
\frac{\partial A(t)}{\partial t}=\sqrt{-1}(\bar{\partial}_A-\partial_A)\Lambda_{\omega}F_A,\\
A(0)=(\bar{\partial}_E, H_0).
\end{cases}
\end{equation}

The global existence and uniqueness of (\ref{equ2}) has been given in \cite{NZ3}. In fact, $A(t)=\sigma(t)(A_0)$, where $\sigma(t)\in \mathcal{G}^{\mathbb{C}}$ satisfies $\sigma(t)^{*H_0}\sigma(t)=H_0^{-1}H(t)$ and $H(t)$ is the long time solution of (\ref{equ1}). It is easy to check the following relations:

\begin{equation}\label{eq9}
\begin{split}
&F_{A(t)}=\sigma(t)\circ F_{H(t),\bar{\partial}_E}\circ \sigma(t)^{-1},\\
&|F_{A(t)}|^2_{H_0}=|F_{H(t),\bar{\partial}_E}|^2_{H(t)}.
\end{split}
\end{equation}
Along the flow (\ref{eq9}), we have the following Bochner type inequality
 \begin{equation}\label{equ3}
  (\triangle_g-\frac{\partial }{\partial t})|F_A|_{H_0}^2\geq 2|\nabla_A F_A|_{H_0}^2-C_1(1+|F_A|_{H_0}+|Ric|_g)|F_A|_{H_0}^2-C_1|F_A|_{H_0}|\nabla_A F_A|_{H_0},
\end{equation}
where $C_1>0$ depends on the geometry of $(X,\omega)$. And denoting
$$\YM(t)=\int_X |F_A|^2_{H_0}\frac{\omega^{n}}{n!},$$
we have the energy inequality

\begin{lemma}\mbox{\rm (\cite[Lemma 2.3]{NZ3})}\label{lmm1}
Let  $(X,\omega)$ be an $n$-dimensional compact Hermitian manifold with $\omega$ satisfying $\partial\bar{\partial}\omega^{n-1}=\partial\bar{\partial}\omega^{n-2}=0$.  Suppose $A(t)$ is a solution of the heat flow (\ref{equ2}) with initial data $A_0$. Then
\begin{equation}
\mbox{\rm YM}(t)+2\int_0^t\int_X\left|\frac{\partial A}{\partial t}\right|^2=\YM(0).
\end{equation}
\end{lemma}

 Let $i_X$ be the infimum of the injective radius over $X$. For any $(x_0, t_0)\in X\times \mathbb{R}^+$ and $r\leq i_X,$ denote $P_r(x_0,t_0)=B_r(x_0)\times [t_0-r^2, t_0+r^2].$ We have the small energy regularity theorem

\begin{lemma}\mbox{\rm (\cite[Theorem 2.10]{NZ3})}\label{lmm4}
  Suppose that $A(t)$ is a smooth solution of the heat flow (\ref{equ2}) over $(X,\omega)$ with $\omega$ satisfying $\partial\bar{\partial}\omega^{n-1}=\partial\bar{\partial}\omega^{n-2}=0$. Then there exist positive constants $\epsilon_0$ and $\delta_0$  depending on the geometry of $(X,\omega)$ and ${\rm YM}(0)$, such that
 if for some $0<R<\min\{i_X/2, \sqrt{t_0}/2\}$, the inequality
$$ R^{2-2n}\int_{P_R(x_0,t_0)}|F_A|^2_{H_0}<\epsilon_0$$
holds, then  for any $\delta\in (0,\min\{\delta_0,1/4\})$, we have
\begin{eqnarray*}
  \sup\limits_{P_{\delta R}(x_0,t_0)}|F_A|^2_{H_0}<16(\delta R)^{-4}.
\end{eqnarray*}
\end{lemma}

%\begin{lemma}\label{lmm2}
%Let $(X,\omega)$ be an $n$-dimensional compact Gauduchon manifold and $(E,\bar{\partial}_E)$ a rank $r$ holomorphic vector bundle. Then $(E,\bar{\partial}_E)$ is $\omega$-semi-stable if and only if it admits an approximate Hermitian Einstein structure.
%\end{lemma}

%\begin{theorem}
%Let $(X,\omega)$ be an $n$-dimensional compact Hermitian manifolds with $\partial\bar{\partial}\omega^{n-1}=\partial\bar{\partial}\omega^{n-2}=0$ and $(E,\bar{\partial}_E)$ be a holomorphic vector bundle over $X.$ If $E$ is semistable and $ch_1(E,H)\cdot [\omega]^{n-1}=ch_2(E,H)\cdot[\omega]^{n-2}=0$, then $(E,\bar{\partial}_E)$ is approximate Hermitian flat.
%\end{theorem}

{\bf Proof of Theorem \ref{thm1}:}

Since $ch_1(E,H)\cdot [\omega^{n-1}]=0,$ we have $\lambda=0.$  Fix a proper Hermitian metric $K$ on $(E,\bar{\partial}_E)$ with ${\rm tr}(\sqrt{-1}\Lambda_{\omega}F_{K,\bar{\partial}_E})=0$ and consider the following perturbed equation
\begin{equation}\label{equ10}
\sqrt{-1}\Lambda_{\omega}F_{H_{\varepsilon},\bar{\partial}_E}+\varepsilon \log f_{\varepsilon}=0,\ \ \ \ \varepsilon\in (0,1],
\end{equation}
where $f_{\varepsilon}=K^{-1}H_{\varepsilon}.$ It has been proved in \cite{LY87,MA} that (\ref{equ10}) is solvable for all $\varepsilon\in (0,1].$
Then for every $\varepsilon$, we consider the following Hermitian-Yang-Mills flow with $H_{\varepsilon}$ as the initial metric
 \begin{equation}
\begin{cases}
H_{\varepsilon}(t)^{-1}\frac{\partial H_{\varepsilon}(t)}{\partial t}=-2\sqrt{-1}\Lambda_{\omega}F_{H_{\varepsilon}(t),\bar{\partial}_E},\\
H_{\varepsilon}(0)=H_{\varepsilon},
\end{cases}
\end{equation}
and its gauge equivalent flow
\begin{equation}\label{equ7}
\begin{cases}
\frac{\partial A_{\varepsilon}(t)}{\partial t}=\sqrt{-1}(\bar{\partial}_{A_{\varepsilon}(t)}-\partial_{A_{\varepsilon}(t)})\Lambda_{\omega}F_{A_{\varepsilon}(t)},\\
A_{\varepsilon}(0)=(\bar{\partial}_E,H_\varepsilon).
\end{cases}
\end{equation}
Since $(E,\bar{\partial}_E)$ is semi-stable, by Lemma \ref{lmm3}, we have
\begin{equation*}
\sup_X|\sqrt{-1}\Lambda_{\omega}F_{H_{\varepsilon},\bar{\partial}_E}|_{H_{\varepsilon}}\rightarrow 0, \ \ \ \ \ \mbox{\rm as }\ \ \varepsilon\rightarrow 0.
\end{equation*}
By (\ref{equ6}),  (\ref{equ9}) and Lemma \ref{lmm1},  we have

\begin{align}
\int_X |F_{A_{\varepsilon}(t)}|^2_{H_{\varepsilon}}\frac{\omega^n}{n!}&\leq\int_X |F_{A_{\varepsilon}(0)}|^2_{H_{\varepsilon}}\frac{\omega^n}{n!}= \int_X |F_{H_{\varepsilon},\bar{\partial}_E}|^2_{H_{\varepsilon}}\frac{\omega^n}{n!}\\
&= \int_X|\sqrt{-1}\Lambda_{\omega}F_{H_{\varepsilon},\bar{\partial}_E}|^2_{H_{\varepsilon}}\frac{\omega^n}{n!}\rightarrow 0,\ \ \varepsilon\rightarrow 0.
\end{align}

%\begin{equation*}
%\leq \int_X |F_{A_{\varepsilon}(0)}|^2_{H_{\varepsilon}}\frac{\omega^n}{n!}\rightarrow 0.
%\end{equation*}

This implies for every $\epsilon>0,$ there exists $\varepsilon(\epsilon)>0,$ such that when $\varepsilon <\varepsilon(\epsilon),$ it holds
\begin{equation}\label{equ11}
\int_X |F_{A_{\varepsilon}(t)}|^2_{H_{\varepsilon}}\frac{\omega^n}{n!}\leq \int_X|F_{A_{\varepsilon}(0)}|_{H_{\varepsilon}}\frac{\omega^n}{n!}<\epsilon.
\end{equation}
Particularly,  there exists $\varepsilon_0>0,$ such that when $\varepsilon<\varepsilon_0,$ it holds
$$\int_X|F_{A_{\varepsilon}(0)}|^2_{H_{\varepsilon}}\frac{\omega^n}{n!}<1.$$

So by the small energy regularity theorem (Lemma \ref{lmm4}), there exist uniform positive constants $\epsilon_0$ and $\delta_0$ depending only on the geometry of $(X,\omega)$, such that for any $(x_0,t_0)\in X\times \mathbb{R}^+$,  if for some $0<R<\min\{i_X/2,\sqrt{t_0}/2\}$,

\begin{equation*}
R^{2-2n}\int_{P_R(x_0,t_0)}|F_{A_{\varepsilon}(t)}|_{H_{\varepsilon}}^2<\epsilon_0
\end{equation*}
holds, then for any $\delta\in(0, \min\{\delta_0,1/4\}),$  we have
$$\sup_{P_{\delta R}(x_0,t_0)}|F_{A_{\varepsilon}(t)}|_{H_{\varepsilon}}^2<16(\delta R)^{-4}.$$
In addition, by (\ref{equ11}), setting  $\epsilon=\frac{1}{2}(\frac{i_X}{2})^{2n-4}\epsilon_0,$ we can find a positive constant $\varepsilon_1$, such that when $\varepsilon<\varepsilon_1$, it holds

 \begin{equation*}
 \int_X |F_{A_{\varepsilon}}|^2_{H_{\varepsilon}}(x,t)<\frac{1}{2}\left(\frac{i_X}{2}\right)^{2n-4}\epsilon_0,
 \end{equation*}
 for any $t>0.$
 Choose $R=i_X/2$. We have for any $x_0\in X$ and $\tau\geq i_X^2$, when $\varepsilon<\varepsilon_1$,  it holds
 \begin{equation}
R^{2-2n}\int_{P_R(x_0,\tau)} |F_{A_{\varepsilon}}|^2_{H_{\varepsilon}}(x,t)\leq R^{2-2n}\int_{\tau-R^2}^{\tau+R^2} \int_X|F_{A_{\varepsilon}}|^2_{H_{\varepsilon}}(x,t)\leq 2R^{4-2n}\times (\frac{1}{2}\epsilon_0 R^{2n-4})=\epsilon_0.
\end{equation}
Then from the small energy regularity theorem, we have
\begin{equation}
|F_{A_{\varepsilon}}|^2_{H_{\varepsilon}}(x_0, \tau)\leq \sup_{P_{\delta_0/2R}(x_0,\tau)}|F_{A_{\varepsilon}}|^2_{H_{\varepsilon}}(x, t)\leq 16\left(\frac{1}{2}\delta_0R\right)^{-4}=256(\delta_0R)^{-4}, \ \ \ \varepsilon<\varepsilon_1.
\end{equation}
This implies that when $\varepsilon<\varepsilon_1$,  $\sup_X|F_{A_{\varepsilon}}|_{H_{\varepsilon}}^2(\cdot, t)$ is uniformly bounded in $[i_X^2,\infty).$ From the Bochner type inequality (\ref{equ3}), when $\varepsilon<\varepsilon_1$, there exists a positive constant $C_2$ independent of $\varepsilon,$ such that
\begin{equation}
  (\triangle_g-\frac{\partial }{\partial t})|F_{A_{\varepsilon}}|_{H_{\varepsilon}}^2\geq -C_2|F_{A_{\varepsilon}}|_{H_{\varepsilon}}^2.
\end{equation}
Using the parabolic mean value inequality, we can find a positive constant $C_3$ independent of $\varepsilon$($\varepsilon<\varepsilon_1$), such that  for $t>i_X^2+1,$ it holds
\begin{equation}
\begin{split}
\sup_X|F_{A_{\varepsilon}}|^2_{H_{\varepsilon}}(\cdot, t)
&\leq C_3\int_X |F_{A_{\varepsilon}}|^2_{H_{\varepsilon}}(x,t-1)\leq C_3 \int_X |F_{A_{\varepsilon}}|^2_{H_{\varepsilon}}(\cdot,0)\\
&=C_3\int_X |F_{H_{\varepsilon},\bar{\partial}_E}|^2_{H_{\varepsilon}}=C_3\int_X |\sqrt{-1}\Lambda_{\omega}F_{H_{\varepsilon},\bar{\partial}_E}|_{H_{\varepsilon}}^2.
\end{split}
\end{equation}

Choosing  $t=i_X^2+2$, from the above, we have when $\varepsilon<\varepsilon_1,$
\begin{equation}
\sup_X|F_{H_{\varepsilon}(i_X^2+2),\bar{\partial}_E}|^2_{H_{\varepsilon}(i_X^2+2)}= \sup_X |F_{A_{\varepsilon}}|^2_{H_{\varepsilon}}(\cdot, i_X^2+2)\leq C_3\int_X |\sqrt{-1}\Lambda_{\omega}F_{H_{\varepsilon}}|_{H_{\varepsilon}}^2\rightarrow 0,\ \ \ \varepsilon \rightarrow 0.
\end{equation}
This implies  the existence of approximate Hermitian flat structure on semistable vector bundles with $ch_1(E,H)\cdot [\omega^{n-1}]=ch_2(E,H)\cdot[\omega^{n-2}]=0$.\\

\qed

\section{Proof of Theorem \ref{thm2}}
 In this section, we will give a detailed proof of Theorem \ref{thm2}.

 %----------------------------------------------------------------------------------------------
%{\bf The second Chern number of nflat bundle }\\
%In this section, we assume that $\omega$ satisfies $\partial{\bar{\partial}}\omega^{n-1}=\partial\bar{\partial}\omega^{n-2}=0.$

%Given two Hermitian metrics $H$ and $K$ on $E$, there exist $\Phi_l\in A^{2k}$, such that
%$$c_l(E,H)-c_l(E,K)=\bar{\partial}\partial\Phi_l,$$
%where $l=1,2,\cdots.$ For instance,
%\begin{equation}
%\begin{split}
%&\Phi_1=\frac{i}{2\pi}\log \det e^s;\\
%&\Phi_2=2\int_0^1{\rm{tr}}(sF_{Ke^{ts}})dt-\log\det e^s\bar{\partial}\partial \log\det e^s,
%\end{split}
%\end{equation}
%where $s=\log K^{-1}H.$
%So when the K\"ahler form $\omega$ satisfies $\partial{\bar{\partial}}\omega^{n-1}=\partial\bar{\partial}\omega^{n-2}=0$, the first  and the second Chern number
%\begin{equation}
%\begin{split}
%& \int_X c_1(E,H)\wedge \frac{\omega^{n-1}}{(n-1)!} \\
% &\int_X c_2(E,H)\wedge \frac{\omega^{n-1}}{(n-1)!}
 %\end{split}
% \end{equation}
% are well-defined.

% \begin{proposition}
 %Let $(E,\bar{\partial}_E)$ be a holomorphic vector bundle of rank $r$ over $X$.  Assume $E$ is nflat, then it holds that
 %\begin{equation}
 %\int_X c_2(E,H)\wedge \frac{\omega^{n-1}}{(n-1)!}=0.
%\end{equation}
 %\end{proposition}

{\bf Proof of Theorem \ref{thm2}:}

We first prove that $(1)\Rightarrow (2)$. We need only to show $ch_1(E,H)\cdot [\omega^{n-1}]=0.$ Since $(E,\bar{\partial}_E)$ is nflat, by Proposition \ref{prop5}, we have $\det E$ is Hermitian flat. This implies
\begin{align}
\label{equ12}ch_1(E,H)\cdot[\omega^{n-1}]= c_1(E)\cdot [\omega^{n-1}]=c_1(\det E)\cdot [\omega^{n-1}]=0.
\end{align}

\medskip

Then we prove $(2)\Rightarrow (3).$
Equality (\ref{equ12}) implies $\deg_{\omega} E=0.$ For any proper saturated sub-sheaf $\mathcal{F}\hookrightarrow E$, $E/\mathcal F$ is torsion free and $(E/\mathcal F)^*$ is a proper subsheaf of $E^*$. Following the argument in \cite[Theorem 1.18]{DPS}, it holds
\begin{equation}
\deg_{\omega} (E/\mathcal F)^*\leq 0.
\end{equation}
Together with $\deg_{\omega} E=0$, we have
\begin{equation*}\deg_{\omega}\mathcal F\leq 0.
\end{equation*}
This implies $(E,\bar{\partial}_E)$ is semi-stable. Then it remains to show $ch_2(E,H)\cdot[\omega^{n-2}]=0$.
Since $(E,\bar{\partial}_E)$ is nef with $ch_1(E,H)\cdot [\omega^{n-1}]=0,$  $\det E$ is nef with $ch_1(\det E)\cdot [\omega^{n-1}]=0$. By Lemma \ref{nefline}, we have $\det E$ is Hermitian flat. This implies

\begin{align}
\label{c12} c_1(E)^2\cdot [\omega^{n-2}]=c_1(\det E)^2\cdot [\omega^{n-2}]=0.
\end{align}

Let $h_1$ and $h_2$ be two Hermitian metrics on $\mathcal{O}_E(1)$. It is easy to check that
 \begin{equation}
 \Theta(\mathcal{O}_E(1),h_1)-\Theta(\mathcal{O}_E(1),h_2)=\bar{\partial}\partial \log h_2^{-1}h_1.
 \end{equation}

 Since $\partial\bar{\partial}\omega^{n-1}=\partial\bar{\partial}\omega^{n-2}=0$ and $h^{-1}_{2}h_1$ is a well-defined smooth function,
 \begin{equation}
 \begin{split}
 &\int_{PE} \Theta(\mathcal{O}_E(1),h_1)^r\wedge \pi^*{\omega^{n-1}}-\int_{PE} \Theta(\mathcal{O}_E(1),h_2)^r\wedge \pi^*{\omega^{n-1}}\\
 =&\int_{PE}\sum_{i=1}^r \left(\begin{array}{c}r \\ i \end{array}\right)  (\bar{\partial}\partial \log h_2^{-1}h_1)^i\wedge \Theta(\mathcal{O}_E(1),h_2)^{r-i} \wedge \pi^*{\omega^{n-1}}\\
 =&0
 \end{split}
 \end{equation}
 and
 \begin{equation}
 \begin{split}
 &\int_{PE} \Theta(\mathcal{O}_E(1),h_1)^{r+1}\wedge \pi^*{\omega^{n-2}}-\int_{PE} \Theta(\mathcal{O}_E(1),h_2)^{r+1}\wedge \pi^*{\omega^{n-2}}\\
 =& \int_{PE}\sum_{i=1}^{r+1} \left(\begin{array}{c}r+1 \\ i \end{array}\right)  (\bar{\partial}\partial \log h_2^{-1}h_1)^i\wedge \Theta(\mathcal{O}_E(1),h_2)^{r+1-i} \wedge \pi^*{\omega^{n-2}}\\
 =&0.
 \end{split}
 \end{equation}

 This implies that $\int_{PE} \Theta(\mathcal{O}_E(1),h)^r \wedge \pi^*{\omega^{n-1}}$ and $\int_{PE} \Theta(\mathcal{O}_E(1),h)^{r+1}\wedge \pi^*{\omega^{n-2}}$ are independent of  the choice of Hermitian metrics on $\mathcal{O}_{E}(1)$.

 Endow $PE$ with a Hermitian metric $\omega_{PE}$. Since $(E,\bar{\partial}_E)$ is nflat, $\mathcal{O}_E(1)$ is nef. This means for every $\epsilon>0$, there exists a Hermitian metric $h_{\epsilon}$ on $\mathcal{O}_E(1)$, such that
\begin{equation}
\sqrt{-1}\Theta(\mathcal{O}_E(1),h_{\epsilon})\geq -\epsilon \omega_{PE}.
\end{equation}
So
\begin{equation*}
\begin{split}
0&\leq \int_{PE}\left(\sqrt{-1}\Theta(\mathcal{O}_E(1),h_{\epsilon})+\epsilon\omega_{PE}\right)^{r+1}\wedge \pi^*{\omega^{n-2}}\\
 &=\int_{PE}(\sqrt{-1}\Theta(\mathcal{O}_E(1), h_{\epsilon}))^{r+1}\wedge \pi^*{\omega^{n-2}}\\
 &+\sum_{i=1}^r \left(\begin{array}{c}r+1 \\ i \end{array}\right) (\sqrt{-1}\Theta( \mathcal{O}_E(1), h_{\epsilon}))^{r+1-i}\wedge(\epsilon \omega_{PE})^i\wedge \pi^*\omega^{n-2}\\
 &\rightarrow \int_{PE} (\sqrt{-1}\Theta(\mathcal{O}_E(1),h))^{r+1}\wedge \pi^*{\omega^{n-2}}, \ \ \ \ \mbox{as} \ \ \epsilon\rightarrow 0.
\end{split}
\end{equation*}
 So by Lemma \ref{lmm5},
 \begin{equation}\label{equ10}
s_2(E)\cdot [\omega^{n-2}]=\int_{PE}\Xi^{r+1}\wedge \pi^*\omega^{n-2}\geq 0,
\end{equation}
where $\Xi=\frac{\sqrt{-1}}{2\pi}\Theta(\mathcal{O}_E(1),h)$ and $h$ is an arbitrary metric on $\mathcal{O}_E(1).$
 From (\ref{c12}), (\ref{equ10}) and (\ref{segre2}),  we have
 \begin{equation}\label{equ13}
 c_2(E)\cdot [\omega^{n-2}]=c_1(E)^2\cdot[\omega^{n-2}]-s_2(E)\cdot[\omega^{n-2}]\leq 0.
 \end{equation}
On the other hand, since $E$ is semi-stable, by Bogomolov inequality (Proposition \ref{Bogomolov}), we have
 \begin{equation}\label{equ14}
 c_2(E)\cdot[\omega^{n-2}]\geq \frac{r-1}{2r}c_1(E)^2\cdot[\omega^{n-1}]= 0.\end{equation}
Combining  (\ref{equ13}) and (\ref{equ14}), we have
 \begin{equation}c_2(E)\cdot[\omega^{n-2}]=0\end{equation}
and consequently $ch_2(E,H)\cdot [\omega^{n-2}]=\frac{1}{2}(c_1^2(E)-2c_2(E))\cdot [\omega^{n-2}]=0.$

 \medskip

 $(3)\Rightarrow (4)$ is just Theorem \ref{thm1}.

 \medskip

  $(4)\Rightarrow (5)$.  This can be proved by the result of the existence of Harder-Narasimhan filtration on non-K\"ahler manifolds (\cite{Bru}) and the argument of Step $2$ and Step $3$ in the proof of Theorem 1.1 in \cite{NZ1}. Here we omit the proof.

 \medskip

At last,  we prove  $(5)\Rightarrow (1)$. It is obvious that Hermitian flat vector bundles are nflat. And by Proposition \ref{prop2}, we get that $(5)$ implies $(1)$.
\qed

%-----------------------------------------------------------------------------
\bibliographystyle{plain}

\end{document}